\documentclass[11pt]{article}
\usepackage{latexsym,amsmath,color,amsthm,amssymb,epsfig,graphicx,mathrsfs}
\usepackage{graphicx}
\usepackage{amssymb}
\usepackage[left=1in,top=1in,right=1in,bottom=1in]{geometry}
\usepackage[linktocpage=true]{hyperref}
\usepackage{setspace}
\usepackage{amssymb, amsmath, amsthm, graphicx,mathrsfs}
\usepackage{caption}

\usepackage{comment}
\def\qed{\hfill\ifhmode\unskip\nobreak\fi\quad\ifmmode\Box\else\hfill$\Box$\fi}
\def\ite#1{\hfill\break${}$\hbox to 50pt {\quad(#1)\hfill}}

\def\cG{{\mathcal G}}
\def\cH{{\mathcal H}}

\newtheorem{thm}{Theorem}[section]
\newtheorem{cor}[thm]{Corollary}
\newtheorem{const}[thm]{Example}

\newtheorem{lem}[thm]{Lemma}

\newtheorem{lemma}[thm]{Lemma}
\newtheorem{conj}[thm]{Conjecture}

\newtheorem{claim}[thm]{Claim}

\usepackage[parfill]{parskip} 

\begingroup
    \makeatletter
    \@for\theoremstyle:=definition,remark,plain\do{%
        \expandafter\g@addto@macro\csname th@\theoremstyle\endcsname{%
            \addtolength\thm@preskip\parskip
            }%
        }
\endgroup

\begin{document}


\title{\vspace{-0.5in} Conditions for a bigraph to be super-cyclic  }

\author{
{{Alexandr Kostochka}}\thanks{
\footnotesize {University of Illinois at Urbana--Champaign, Urbana, IL 61801
 and Sobolev Institute of Mathematics, Novosibirsk 630090, Russia. E-mail: \texttt {kostochk@math.uiuc.edu}.
 Research 
is supported in part by NSF grant  DMS-1600592
and grants  18-01-00353A and 19-01-00682
  of the Russian Foundation for Basic Research.
}}
  \and{Mikhail Lavrov\thanks{Department of Mathematics, University of Illinois at Urbana--Champaign, IL, USA,  mlavrov@illinois.edu.}}
\and{{Ruth Luo}}\thanks{University of Califonia, San Diego, La Jolla, CA 92093, USA and University of Illinois at Urbana--Champaign, Urbana, IL 61801, USA. E-mail: {\tt ruluo@ucsd.edu}.
Research 
is supported in part by NSF grant  DMS-1902808.
}
\and{{Dara Zirlin}}\thanks{University of Illinois at Urbana--Champaign, Urbana, IL 61801, USA. E-mail: {\tt zirlin2@illinois.edu}.
Research 
is supported in part by Arnold O. Beckman Research Award (UIUC) RB20003.}}

\date{\today}

\maketitle

\vspace{-0.3in}

\begin{abstract}
A hypergraph $\cH$ is
  {\em super-pancyclic} if  for each  $A \subseteq V(\cH)$ with $|A| \geq 3$, $\cH$ contains a Berge cycle
with base vertex set $A$. 
We present two natural necessary conditions for a hypergraph to be super-pancyclic, and  show that
in several classes of hypergraphs  these necessary conditions are also sufficient. In particular, they are sufficient
for every  hypergraph $\cH$ with $ \delta(\cH)\geq \max\{|V(\cH)|, \frac{|E(\cH)|+10}{4}\}$.

We also consider  {\em super-cyclic} bipartite graphs: those are $(X,Y)$-bigraphs $G$ such that for each $A \subseteq X$ with $|A| \geq 3$,
$G$ has a cycle $C_A$ such that $V(C_A)\cap X=A$. Such graphs are incidence graphs of super-pancyclic hypergraphs, and our proofs  use the language of such graphs.
 
\end{abstract}

\section{Introduction}

\subsection{Longest cycles in bipartite graphs and hypergraphs}

For positive integers $n, m,$ and $\delta$ with $\delta \leq m$, let $\cG(n,m,\delta)$ denote the set of all bipartite graphs with a bipartition $(X, Y)$ such that $|X| = n\geq 2, |Y|=m$ and for every $x \in X$, $d(x) \geq \delta$. In 1981, Jackson~\cite{jackson} proved that if  $ \delta\geq \max\{n,\frac{m+2}{2}\}$, then every graph $G\in \cG(n,m,\delta)$ contains a cycle of length $2n$, i.e., a cycle that contains all vertices of $X$.
This result is sharp. 
 Jackson also conjectured that if  $G \in \cG(n,m,\delta)$ is 2-connected, then the lower bound on $\delta$ can be weakened. 

\begin{conj}[Jackson~\cite{jackson}]\label{jacksonconj} Let $m,n,\delta$ be integers. If $\delta\geq \max\{n, \frac{m+5}{3}\}$,
then every $2$-connected graph $G \in \cG(n,m,\delta)$ contains a cycle of length $2n$.\end{conj} 

Recently, the conjecture was proved in~\cite{KLZ}. The restriction $\delta\geq \frac{m+5}{3}$ cannot be weakened any further because of the following example:

\begin{const}\label{con5}
Let $n_1 \geq n_2 \geq n_3\geq 1$ be such that $n_1 + n_2 + n_3 = n$.
 Let $G_3(n_1, n_2, n_3) \in \cG(n,3\delta-4, \delta)$ be the bipartite graph obtained from $K_{\delta-2, n_1} \cup K_{\delta-2, n_2} \cup K_{\delta-2, n_3}$  by adding two vertices $a$ and $b$ that are both adjacent to every vertex in the parts of size $n_1, n_2$, and $n_3$. Then a longest cycle in $G_3(n_1, n_2, n_3)$ has length $2(n_1+n_2) \leq 2(n-1)$.  
 \end{const}
 
Very recently~\cite{KLMZ}, the bound was refined for $3$-connected graphs in
  $G \in \cG(n,m,\delta)$.

\begin{thm}[\cite{KLMZ}]\label{jackson6} 
Let $m,n,\delta$ be integers. If $\delta\geq \max\{n, \frac{m+10}{4}\}$,
then every $3$-connected graph $G \in \cG(n,m,\delta)$ contains a cycle of length $2n$.
\end{thm}

A construction very similar to Construction~\ref{con5} shows that the bound $\frac{m+10}{4}$ is sharp.

The results can be translated into the language of hypergraphs and hamiltonian Berge cycles.



Recall that a {\em hypergraph} $\cH$ is a set of vertices $V(\cH)$ and a set of edges $E(\cH)$ such that each edge is a subset of $V(\cH)$. 
 We consider hypergraphs with  edges of any size. The degree, $d(v)$, of a vertex $v$  is the number of edges that contain $v$.

A {\em Berge cycle} of length $\ell$ in a hypergraph is a set of $\ell$ distinct vertices $\{v_1, \ldots, v_\ell\}$ and $\ell$ distinct edges $\{e_1, \ldots, e_\ell\}$ such that for every $i\in [\ell]$, $v_i, v_{i+1} \in e_i$ (indices are taken modulo $\ell$). The vertices $v_1, \ldots, v_\ell$ are  the {\em base vertices} of the  cycle.

Naturally, a {\em hamiltonian Berge cycle} in  a hypergraph $\cH$ is a Berge cycle whose set of base vertices is $V(\cH)$.

Let $\cH = (V(\cH), E(\cH))$ be a hypergraph. The {\em incidence   graph} of $\cH$ is the bipartite graph $I(\cH)$ with parts $(X, Y)$ where $X = V(\cH)$, $Y= E(\cH)$ such that for $e \in Y, v \in X,$ $ev\in E(I(\cH))$ if and only if the vertex $v$ is contained in the edge $e$ in $\cH$. 

If $\cH$ has $n$ vertices, $m$ edges and minimum degree at least $\delta$, then
 $I(\cH) \in \cG(n,m,\delta)$. There is a simple relation between the cycle lengths in a hypergraph $\cH$ and its incidence graph $I(\cH)$:
 If $\{v_1, \ldots, v_\ell\}$ and $\{e_1, \ldots, e_\ell\}$ form a Berge cycle of length $\ell$ in $\cH$, then $v_1 e_1 \ldots v_\ell e_\ell v_1$ is a cycle of length $2\ell$ in $I(\cH)$, and vice versa.

\subsection{Super-pancyclic hypergraphs and super-cyclic bigraphs}
Recall that an $n$-vertex graph is  {\em pancyclic} if it contains a cycle of length $\ell$ for every $3 \leq \ell \leq n$. 
There are a number of interesting results on pancyclic graphs, see e.g. survey~\cite{pansurvey}.
A similar notion  for hypergraphs and a strengthening of it were recently considered in~\cite{KLZ}.

A hypergraph $\cH$ is {\em pancyclic} if it contains a Berge cycle of length $\ell$ for every $\ell \geq 3$. Furthermore,  $\cH$ is {\em super-pancyclic} if for every   $A\subseteq V(\cH)$ with $ |A|\geq 3$, $\cH$ has  a Berge cycle whose set of base vertices is  $A$.

While the notion of super-pancyclic graphs is useless, since only complete graphs have this property, the notion for general hypergraphs is nontrivial.
For example, Jackson's proof~\cite{jackson}  that for $ \delta\geq \max\{n,\frac{m+2}{2}\}$, each graph $G\in \cG(n,m,\delta)$  has a cycle of length $2n$ 
yields a stronger statement. In the language of hypergraphs, it implies the following.

\begin{thm}  If $\delta\geq \max\{n, \frac{m+2}{2}\}$, then every   $n$-vertex hypergraph with $m$ edges and minimum degree  at least $\delta$ is super-pancyclic.
\end{thm}

It is interesting to find broader conditions guaranteeing that a hypergraph is super-pancyclic. The notion of super-pancyclicity translates into the language of bipartite graphs as follows.

By an $(X,Y)$-bigraph we mean a bipartite graph $G$ with a specified ordered bipartition $(X,Y)$. An $(X,Y)$-bigraph is \emph{super-cyclic} if for every $X' \subseteq X$ with $|X'| \geq 3$, $G$ has a cycle $C$  with $V(C) \cap X = X'$; we say that {\em $C$ is based on  $X'$}.

To state necessary conditions for an $(X,Y)$-bigraph to be super-cyclic, we need a new notion. For $A\subseteq X$, the {\em super-neighborhood} $\widehat{N}(A)$ is the set $\{y\in Y\,:\, |N(y)\cap A|\geq 2\}$.

If $G$ is a super-cyclic $(X,Y)$-bigraph, $A \subseteq X$, and $C$ is a cycle based on $A$, let $B = V(C) \cap Y$. Then $B \subseteq \widehat{N}(A)$ and $G[A \cup B]$ is $2$-connected. Since by the expansion lemma, adding a vertex of degree at least $2$ to a 2-connected graph keeps the graph $2$-connected, we conclude that every super-cyclic bipartite graph satisfies:

\begin{equation}
	\label{lll}
	\text{For each $A \subseteq X$ with $|A| \ge 3$: }
	\begin{cases}
		|\widehat{N}(A)|\geq |A|, \text{ and} \\
		G[A\cup \widehat{N}(A)] \text{ is $2$-connected.}
	\end{cases}
\end{equation}
 
 We  conjecture that  these necessary  conditions for a bigraph to be super-cyclic are also sufficient.
 
\begin{conj}\label{mcon}
If $G$ is an $(X,Y)$-bigraph satisfying~\eqref{lll}, then $G$ is super-cyclic.
\end{conj} 

Jaehoon Kim~\cite{jk} observed that to check condition~\eqref{lll}, it is sufficient to verify that $G[A \cup \widehat{N}(A)]$ is 2-connected only when $|A|=3$, though $|\widehat{N}(A)| \geq |A|$ still needs to be checked for all $A$. When $|A| > 3$, if $G[A \cup \widehat{N}(A)]$ is not $2$-connected, there is a subset $A' \subseteq A$ with $|A'| =3$ for which $G[A' \cup \widehat{N}(A')]$ is also not $2$-connected.
 
To give partial support for Conjecture~\ref{mcon}, let us somewhat refine the notion of super-cyclic bigraphs.
 
For an integer $k\geq 3$,  a bipartite graph $G$ with partition $(X,Y)$ is {\em $k$-cyclic} if for every $X' \subseteq X$ with $|X'| = k$, $G$ has a cycle $C$ that is based on  $X'$. If $G$ is $k$-cyclic for all $3 \le k \le |X|$, then it is super-cyclic.

In a series of claims, we prove the following.

\begin{thm}\label{Xk}  If $G$ is an $(X,Y)$-bigraph satisfying~\eqref{lll}, then $G$ is $k$-cyclic for $k = 3, 4, 5, 6$.
\end{thm}

Another result supporting Conjecture~\ref{mcon} was proved in~\cite{KLZ} (in slightly different terms). 

\begin{thm}[\cite{KLZ}]\label{mainpan} Let $\delta\geq \max\{n, \frac{m+5}{3}\}$. If $G \in \cG(n,m,\delta)$ satisfies~\eqref{lll},
then $G$ is super-cyclic.
\end{thm}

We use Theorem~\ref{Xk} and the ideas of the proof of Theorem~\ref{jackson6}  to strengthen Theorem~\ref{mainpan} as follows.

\begin{thm}\label{mainpan3} Let $\delta\geq \max\{n, \frac{m+10}{4}\}$. If $G \in \cG(n,m,\delta)$ satisfies~\eqref{lll}, then $G$ is super-cyclic.
\end{thm}

In terms of hypergraphs, our result is as follows.

\begin{cor}[Hypergraph version of Theorem~\ref{mainpan3}]\label{mainpan2} Let $\delta\geq \max\{n, \frac{m+10}{4}\}$.
 If the incidence graph of
 an $n$-vertex hypergraph $\cH$  with $m$ edges and minimum degree $\delta(\cH)$ satisfies~\eqref{lll},
 then $\cH$ is super-pancyclic.
\end{cor}

We present the main proofs in the language of bipartite graphs. We will say that an $(X,Y)$-bigraph $G$ is {\em critical} if the following conditions hold:
\begin{enumerate}
\item[(a)] $G$ satisfies~\eqref{lll} but is not super-cyclic,
\item[(b)] $\widehat{N}(X) = Y$, and
\item[(c)] for every $X' \subset X$ with $X' \neq X$, $G[X' \cup Y]$ is super-cyclic.
\end{enumerate}
Note that every graph satisfying~\eqref{lll} is either super-cyclic or has a critical subgraph.

Furthermore, we say that a critical $(X,Y)$-bigraph $G$ is {\em saturated} if, after adding any $X,Y$-edge to $G$, the resulting graph is super-cyclic.

In Section~\ref{sec:critical} we prove basic properties of critical bigraphs. Based on this, in Section~\ref{sec:small-k} we prove Theorem~\ref{Xk} for $k=3, 4$, and $5$. In Section~\ref{sec:saturated} we discuss saturated critical graphs, which will be useful in the last two sections. In Section~\ref{sec:k-6} we prove Theorem~\ref{Xk} for $k=6$. In Section~\ref{sec:main-proof} we prove Theorem~\ref{mainpan3}.


\section{Properties of critical bigraphs}\label{sec:critical}

For all $(X,Y)$-bigraphs $G$ below we assume $|X| \ge 3$, since $G$ is trivially super-cyclic when $|X|\le 2$.

\begin{lemma}\label{common-neighbor}
Suppose that an $(X,Y)$-bigraph $G$ satisfies~\eqref{lll}.  Then $|N(x) \cap N(x')| \ge 1$ for all distinct $x,x'\in X$.
\end{lemma}
\begin{proof}
Let $x''$ be any vertex in $X-\{x,x'\}$ and $A = \{x, x', x''\}$. If $N(x) \cap N(x')=\emptyset$, then $G[A \cup \widehat{N}(A)] - x''$ has no $x, x'$-path, contradicting~\eqref{lll}.
\end{proof}

\begin{claim}\label{crit2}
Let 	 $G$ be a critical $(X,Y)$-bigraph. Then 
 $G$ is $2$-connected.	
	\end{claim}
\begin{proof} This is by the fact that $Y=\widehat{N}(X)$ and by~\eqref{lll}.
\end{proof}



Let $G$ be a critical $(X,Y)$-bigraph with $|X|=k+1$ and $x_0\in X$. By  definition, $G-\{x_0\}$ is super-cyclic. In particular, it has a cycle
 $C = x_1y_1x_2y_2\dots x_ky_kx_1$ based on $X-\{x_0\}$.
We index the vertices of $C$ modulo $k$; for example, $x_{k+1} = x_1$. We  derive some properties of such triples $(G,x_0,C)$.

\smallskip
\begin{claim}\label{no-con-1}
	For all  $y_i, y_j \in N(x_0)$, $x_i$ and $x_j$ have no common neighbor outside $C$. Similarly, $x_{i+1}$ and $x_{j+1}$ have no common neighbor outside $C$.
\end{claim}
\begin{proof}
	If $x_i$ and $x_j$ have a common neighbor $y \notin V(C)$, then the cycle
	\[
		x_1 y_1 \dots x_i y x_j y_{j-1} \dots y_i x_0 y_j x_{j+1} \dots x_1
	\]
	is based on  $X$, contrary to assumption. If $x_{i+1}$ and $x_{j+1}$ have such a common neighbor, consider the cycle $C$ in reverse and apply the same argument.
\end{proof}

\begin{claim}\label{no-con-2}
	For every $y_i \in N(x_0)$,  $x_i$ and $x_0$ have no common neighbor outside $C$; similarly,  $x_{i+1}$ and $x_0$ have no common neighbor outside $C$.
\end{claim}
\begin{proof}
	If $x_i$ and $x_0$ have a common neighbor $y \notin V(C)$, then we may extend $C$ to a cycle based on $X$ by replacing the edge $x_iy_i$ with the path $x_i y x_0 y_i$. The proof for $x_{i+1}$ is similar.
\end{proof}

\begin{claim}\label{no-con-3}
	For every $i$, if $x_i$ has a common neighbor $y$ with $x_0$ outside $C$, then $x_{i+1}$ has no common neighbor with $x_0$ outside $C$, except possibly for $y$.
\end{claim}
\begin{proof}
	If $x_{i+1}$ and $x_0$ have a common neighbor $y' \notin V(C)$, with $y' \ne y$, then we may extend $C$ to a cycle based on $X$ by replacing the path $x_i y_i x_{i+1}$ with the path $x_i y x_0 y' x_{i+1}$.
\end{proof}

\begin{lemma}\label{two-neighbors}
	The vertex $x_0$ has at least two neighbors in $C$.
\end{lemma}
\begin{proof}
Let $A$ be the subset of $X$ consisting of $x_0$, together with all $x_i$ that do not have  a common neighbor with $x_0$ outside $C$.

If $|A| \ge 3$, then $G[A \cup \widehat{N}(A)]$ is $2$-connected by~\eqref{lll}, so $x_0$ has at least two neighbors in $\widehat{N}(A)$. Each of these neighbors must also be adjacent to at least one vertex in $A - \{x_0\}$. By our choice of $A$, these neighbors must be in $C$, and we are done.

If $|A| \le 2$, then $x_0$ has a common neighbor outside $C$ with all but at most one of $x_1, x_2, \dots, x_k$. By Claim~\ref{no-con-3}, two consecutive vertices $x_i, x_{i+1}$ cannot have different common neighbors with $x_0$ outside $C$. Therefore there is a vertex $y_0$ outside $C$ adjacent to $x_0$ and to all but at most one of $x_1, x_2, \dots, x_k$.

By Claim~\ref{crit2}, $d(x_0)\geq 2$. So there are two possibilities:
\begin{itemize}
\item If $x_0$ has a neighbor $y_i$ in $C$, then at least one of $x_i$ or $x_{i+1}$ is adjacent to $y_0$; then it has a common neighbor with $x_0$ outside $C$, contradicting Claim~\ref{no-con-2}.

\item If $x_0$ has a neighbor $y_0'$ outside $C$, then $y_0'$ has a neighbor $x_i$ in $C$ because $\delta(G) \ge 2$. By Claim~\ref{no-con-3}, $x_{i-1}$ and $x_{i+1}$ cannot have common neighbors with $x_0$ outside $C$ except possibly for $y_0'$. However, at least one of them is adjacent to $y_0$, which is  a contradiction.
\end{itemize}
Therefore the case $|A| \le 2$ is impossible, completing the proof.
\end{proof}

\section{3-, 4-, and 5-cyclic graphs}\label{sec:small-k}

Theorem~\ref{Xk} makes four claims: for $k=3, 4, 5, 6$. In this section, we prove three of them.

\smallskip
\begin{claim}\label{x3}
All $(X,Y)$-bigraphs $G$ satisfying~\eqref{lll} are $3$-cyclic.
\end{claim}
\begin{proof}
Suppose the claim is false and take a vertex-minimal counter-example, so that $|X|=3$ and $Y=\widehat{N}(X)$. Then $G$ is critical. 
Suppose the longest cycle $C=x_1y_1x_2y_2x_1$ of $G$ has 4 vertices and does not include the vertex $x_3$ (there is a cycle because $G$ is $2$-connected).
By Claim~\ref{crit2} and the fan lemma,
  there are 2 paths from $x_3$ to $C$ having only  $x_3$ in common. Then $G$ would contain a cycle of length 6 unless the paths are just $x_3y_1$ and $x_3y_2$. By symmetry, for any $y_3$ not in $C$, we need $y_3x_1, y_3x_2 \in E(G)$. Then we get a $6$-cycle $x_1y_1x_3y_2x_2y_3x_1$. 
\end{proof}

\begin{claim}\label{x4}
All  $(X,Y)$-bigraphs $G$ satisfying~\eqref{lll} are $4$-cyclic.
\end{claim}
\begin{proof} Suppose the claim is false and take a vertex-minimal counter-example, so that $|X|=4$ and $Y=\widehat{N}(X)$. Then $G$ is critical.  Let $x_0 \in X = \{x_0,x_1,x_2,x_3\}$ have maximum degree. By Claim~\ref{x3}, $G-\{x_0\}$ has a $6$-cycle $C = x_1y_1x_2y_2x_3y_3x_1$.

\textbf{Case 1:} $x_0$ has a neighbor $y_0$ outside of $C$. By Lemma~\ref{two-neighbors}, $x_0$ is adjacent to at least two of $\{y_1, y_2, y_3\}$; since $\delta(G) \ge 2$, $y_0$ is adjacent to at least one of $\{x_1,x_2,x_3\}$. Then there is an edge of $C$ both of whose endpoints are adjacent to $x_0$ or $y_0$; without loss of generality, it's $x_1y_1$. We can replace $x_1y_1$ by $x_1y_0x_0y_1$, extending $C$.

\textbf{Case 2:} All of $x_0$'s neighbors are in $C$. Then the neighbors of $x$ in $C$ have degree at least $3$, and all other vertices in $Y$ at least $2$. Since $|Y| \ge |X|$, there are also vertices of $X$ with degree $3$, and since $x_0$ was chosen to have maximum degree in $X$, its degree is at least $3$. Therefore $x_0$ is adjacent to all of $\{y_1,y_2,y_3\}$.

Since $|Y|\ge 4$, there is $y_0 \in Y$ outside $C$. By Claim~\ref{crit2}, $y_0$ has at least two neighbors in $X$, and neither of them is $x_0$. Without loss of generality, $y_0$ is adjacent to $x_1$ and $x_2$, and so $G$ has a cycle $x_0y_1x_2y_0x_1y_3x_3y_2x_0$.

In both cases, we get an $8$-cycle,  a contradiction.
\end{proof}

\begin{claim}\label{x5}
All  $(X,Y)$-bigraphs $G$ satisfying~\eqref{lll} are $5$-cyclic.
\end{claim}
\begin{proof}
Suppose the claim is false and take a vertex-minimal counter-example, so that $|X|=5$ and $Y=\widehat{N}(X)$. Then $G$ is critical. Let $x_0 \in X = \{x_0, x_1,x_2, x_3, x_4\}$ have maximum degree. By Claim~\ref{x4}, $G-\{x_0\}$ has an $8$-cycle $C = x_1y_1x_2y_2x_3y_3x_4y_4$.

\textbf{Case 1:} $x_0$ has a neighbor $y_0$ outside of $C$. By Lemma~\ref{two-neighbors}, $x_0$ is adjacent to at least two of $\{y_1, y_2, y_3, y_4\}$; since $\delta(G) \ge 2$, $y_0$ is adjacent to at least one of $\{x_1,x_2,x_3, x_4\}$. In almost all cases, there is an edge of $C$ both of whose endpoints are adjacent to $x_0$ or $y_0$, in which case we are done as before. The remaining case is unique up to relabeling $C$; without loss of generality, $x_0$ is adjacent to $y_1$ and $y_2$ and $y_0$ is adjacent to $x_4$.

If $x_3y_4$ is an edge, then there is a $10$-cycle
$\quad
	x_1y_1x_2y_2x_0y_0x_4y_3x_3y_4x_1,
\quad$
and similarly there is a $10$-cycle if $x_1y_3$ is an edge. If neither is an edge, then $\widehat{N}(\{x_0, x_1, x_3\})$ contains $y_1$ and $y_2$, but not $y_3$ or $y_4$, so it needs a third vertex (call it $y_5$) which is outside $C$, adjacent to $x_1$ and either to $x_0$ or to $x_3$. In either case, we get a $10$-cycle: one of
\[
	x_1y_5x_3y_2x_2y_1x_0y_0x_4y_4x_1 \text{ or } x_1y_5x_0y_1x_2y_2x_3y_3x_4y_4x_1.
\]
Note that in Case 1, we did not use that $x_0$ has maximum degree.

\textbf{Case 2:} All of $x_0$'s neighbors are in $C$. In this case, as before, we argue that $x_0$ must have degree at least $3$. Say $x_0$ is adjacent to $\{y_1, y_2, y_3\}$; we make no assumption about whether $x_0$ is adjacent to $y_4$.

We can replace $x_2$ or $x_3$ by $x_0$ to get new cycles using the same vertices $y_1,y_2,y_3,y_4$ of $Y$. If $x_2$ or $x_3$ has a neighbor other than $y_1,y_2,y_3,y_4$, then we can apply Case 1.

So all the other vertices of $Y$ (and there must be at least one) must be adjacent only to $x_1$ and $x_4$. Since they can be swapped in for $y_4$ to get a new cycle, if $y_4$ is adjacent to any of $x_0, x_2, x_3$, we can also reduce to a cycle $C$ where Case~1 applies. Therefore $y_4$ is also adjacent only to $x_1$ and $x_4$.

But now $\widehat{N}(\{x_0, x_1, x_2, x_3\}) = \{y_1, y_2, y_3\}$ which violates~\eqref{lll}. 
In all cases, we get a contradiction.
\end{proof}

\section{Saturated critical bigraphs}\label{sec:saturated}

Recall that a critical $(X,Y)$-bigraph $G$ is {\em saturated} if adding to $G$ any $X,Y$-edge results in a super-cyclic bigraph. 

\begin{lemma}\label{5c} If $G$ is a saturated critical $(X,Y)$-bigraph, then for every $y\in Y$, $|N(y)|\neq |X|-1$.
\end{lemma}

\begin{proof} Suppose $G$ is a saturated critical $(X,Y)$-bigraph,
 and for $y_0\in Y$ and $x_0\in X$ we have $N(y_0)=X-\{x_0\}$.
 Since $G$ is critical, $G-\{x_0\}$ is super-cyclic, but $G$ has no cycles based on $X$. Let $|X|=k$.
 Since $G$ is saturated,  $G+y_0x_0$ has a $2k$-cycle $y_0x_1y_1x_1y_2\ldots x_{k}y_{0}$ where $x_{k}=x_0$.
 Then $G$ contains path $P=y_0x_1y_1x_2\ldots x_{k}$.

By the choice of $y_0$, $\{x_1,\ldots,x_{k-1}\}\subseteq N(y_{0})$. Thus if $x_k$ is adjacent to any $y_j$ for $1\leq j\leq k-2$,
then $G$ has  cycle $x_ky_jx_jy_{j-1}\ldots y_0x_{j+1}y_{j+1}\ldots x_k$, a contradiction.
Hence $x_k$ has only one neighbor on $P$. Let $N_{G}(x_k)=\{y_{k-1},z_1,z_2,\ldots,z_s\}$.
Since $G$ is $2$-connected, $s\geq 1$. Again, if any $z_i$ is adjacent to any $x_j$ for $j\leq k-2$, then
$G$ has cycle $x_kz_ix_jy_{j-1}\ldots y_0x_{j+1}y_{j+1}\ldots x_k$, a contradiction. Hence $N(z_i)=\{x_{k-1},x_{k}\}$ for
all $1\leq i\leq s$. Switching $z_1$ with $y_{k-1}$ we conclude that $N(y_{k-1})=\{x_{k-1},x_{k}\}$. 
So, the only vertex of $X-\{x_k\}$ at distance $2$ from $x_k$ is $x_{k-1}$, a contradiction to Lemma~\ref{common-neighbor}.
\end{proof}


\begin{lemma}\label{l5}  If $G$ is a saturated critical $(X,Y)$-bigraph and some $x_0 \in X$ has degree $2$, then
\begin{enumerate}
\item[(a)] each of its neighbors is adjacent to all vertices in $X$, and
\item[(b)] $d(x)\geq 4$ for every $x\in X-\{x_0\}$.
\end{enumerate}
In particular, at most one vertex in $X$ has degree $2$.
\end{lemma}

\begin{proof} Suppose $G$ is a saturated critical $(X,Y)$-bigraph, and $d(x_0)=2$ for some $x_0\in X$.  Let $N(x_0)=\{y_1,y_2\}$. We first prove part (a):
\begin{equation}\label{c1}
	N(y_1) = N(y_2) = X. 
\end{equation}
Indeed, suppose  $N(y_j)\neq X$ for some $j\in \{1,2\}$. Then by Lemma~\ref{5c}, $|X-N(y_j)|\geq 2$,
  say, $\{x,x'\}\subseteq X-N(y_j)$.    Consider $A=\{x_0,x,x'\}$ and $B=\widehat{N}_{G}(A)$. Then $y_j\notin B$ and so $d_{G[A\cup B]}(x_0)\leq 1$,
a contradiction to~\eqref{lll}. 
 This proves~\eqref{c1}.

To prove part (b), consider an $x\in X-\{x_0\}$. For any $x'\in X-\{x,x_0\}$, Claim~\ref{x3} for $A=\{x,x',x_0\}$ yields
that there is a common neighbor $y(x')$ of $x$ and $x'$ distinct from $y_1$ and $y_2$. If (b) does not hold then all $y(x')$ coincide,
and hence there is a vertex $y$ adjacent to all vertices in $X$ apart from $x_0$, a contradiction to  Lemma~\ref{5c}. This proves (b).
\end{proof}

\begin{lemma}\label{-2} If $G$ is a saturated critical $(X,Y)$-bigraph,
then
for every $y\in Y$, $|N(y)|\neq |X|-2$.
\end{lemma}

\begin{proof} Suppose $G$ is a saturated critical $(X,Y)$-bigraph,
 and   $N(y_0)=X-\{x',x''\}$ for some $y_0\in Y$. Let $|X|=k$.


Assume   $d(x')\geq d(x'')$.   By Lemma~\ref{l5}, $d(x')\geq 3$.
Since $G$ is saturated, it
 has a path $P=y_0x_1y_1x_2\ldots y_{k-1}x_{k}$ where $x_{k}=x'$. 
 We may assume $x''=x_j$ for some $j$.

If $x_k$ is adjacent to any $y_i$ for $i\in [k-2]-\{j-1\}$,
then $G$ has  cycle $y_0x_{i+1}y_{i+1}x_{i+2}\ldots x_ky_{i}x_{i}\ldots y_0$, a contradiction. So $N(x_k)\cap V(P)\subseteq \{y_{k-1},y_{j-1}\}$.
 Let $N(x_k)-V(P)=\{z_1,z_2,\ldots,z_s\}$. Since $d(x_k)\geq 3$, $s\geq 1$. Let $T=X- \{x_k,x_{k-1},x_{j-1}\}$.
 Again, if any $z_\ell$ is adjacent to any $x_i\in T$, then
$G$ has cycle $y_0x_{i+1}y_{i+1}x_{i+2}\ldots x_kz_{\ell}x_{i}\ldots y_0$, 
 a contradiction. Hence 
 \begin{equation}\label{193}
	\mbox{\em  $N(z_\ell)\cap T=\emptyset\quad$ for
all $1\leq \ell\leq s$.
	}
\end{equation}
 Since Claim~\ref{x5} implies $k\geq 6$, $|T|\geq 3$. By  Claim~\ref{x3}, for each $x_i,x_{i'}\in T$, $G$ contains a $6$-cycle $C_1$ with $V(C_1)\cap X=
 \{x_k,x_i,x_{i'}\}$, say $C_1=x_kyx_iy'x_{i'}y''x_k$. By~\eqref{193}, $\{y,y''\}\subseteq \{y_{k-1},y_{j-1}\}$. In particular, $x_ky_{j-1}\in E(G)$.
 
 Similarly, if there are $x_i,x_{i'}\in T$ both not adjacent to $y_{k-1}$ or both not adjacent to $y_{j-1}$, then $G$ does not contain a $6$-cycle $C_1$ with $V(C_1)\cap A=
 \{x_k,x_i,x_{i'}\}$; however, $G$ is $3$-cyclic, a contradiction. This means $|N(y_{k-1})\cap T|\geq |T|-1$ and $|N(y_{j-1})\cap T|\geq |T|-1$. Since
 $|T|\geq 3$, this implies that there is $x_i\in T\cap N(y_{k-1})\cap N(y_{j-1})$.

Since $G$ is $2$-connected, $z_1$ has a neighbor in $\{x_{k-1},x_{j-1}\}$. If $z_1x_{k-1}\in E(G)$, then $G$ has  cycle
$y_0x_1\ldots x_i y_{k-1}x_kz_1x_{k-1}y_{k-2}x_{k-2}\ldots x_{i+1}y_0$. So $N(z_1)=\{x_k,x_{j-1}\}$.
If $j\neq k-1$, then $G$ has the cycle
$y_0x_1\ldots x_{j-1}z_1x_k y_{j-1}x_jy_jx_{j+1}\ldots x_{k-1}y_0$. Hence we may suppose $j=k-1$. Then by the definition of $T$,
$i\leq k-3$. So  $G$ has the cycle
\[
	y_0x_1\ldots x_i y_{k-2}x_{k-1}y_{k-1}x_kz_1x_{k-2}y_{k-2}\ldots x_{i+1}y_0,
\]
a contradiction.
\end{proof}

A critical $(X,Y)$-bigraph $G$ is {\em $Y$-minimal} if every proper subgraph $G'=(X',Y';E')$ of $G$ satisfying~\eqref{lll} is super-cyclic.

\begin{lemma}\label{y2}  If  a saturated critical $Y$-minimal $(X,Y)$-bigraph $G$ 
has vertices $y_1, y_2 \in Y$ of degree $2$, then $N(y_1) \neq N(y_2)$.
\end{lemma}

\begin{proof}
Suppose $N(y_1) = N(y_2) = \{x_1, x_2\}$, and consider the graph $G' := G - \{y_1\}$ with partite sets $X$ and $Y'= Y - \{y_1\}$. Note that in $G$, each cycle
of length at least $6$ contains at most one vertex in $\{y_1,y_2\}$ since the neighbors of such a vertex on the cycle must be exactly $x_1$ and $x_2$.  Hence for each cycle $C$ of length at least $6$ in $G$, there exists a cycle $C'$ in $G'$ with $C \cap X = C' \cap X$. 
We will show that~\eqref{lll} holds for $G'$.

Indeed, suppose there exists a set $A \subseteq X$ with $|\widehat{N}_{G'}(A)| < |A|$. Then $\{x_1, x_2\} \subseteq A$, $\widehat{N}_{G'}(A) = \widehat{N}_G(A) - \{y_1\}$, and hence $|\widehat{N}(A)| = |A|$. If $|A| \geq 4$, then $|\widehat{N}(A - \{x_1\})| \geq |A - \{x_1\}| = |A| - 1$. However, $\widehat{N}(A - \{x_1\}) \subseteq \widehat{N}(A) - \{y_1, y_2\}$, a contradiction. So $|A| = 3$, say $A = \{x_1, x_2, x_3\}$, $\widehat{N}_{G'}(A) = \{y_2, y_3\}$, and $\widehat{N}(A) =\{y_1, y_2, y_3\}$. But there is no 6-cycle in $G$ based on $A$ since $N(y_1) = N(y_2) = \{x_1, x_2\}$. This contradicts Claim~\ref{x3}.

Now suppose $G'$ is not 2-connected. Then  $G'$ contains a cut vertex $v$, and $\{v, y_1\}$ is a cut set in $G$. This implies that $x_1$ and $x_2$ are in different components of $G - \{v, y_1\}$, and so $v = y_2$.  Let $x_3 \in X - \{x_1, x_2\}$. Then there is no 6-cycle based on $\{x_1, x_2, x_3\}$ in $G$, a contradiction.

By the definition of critical $Y$-minimal bigraphs, $G'$ is super-cyclic; but then $G$ also is.
\end{proof}

\begin{lemma}\label{C-2}  If $G$ is a saturated critical $Y$-minimal $(X,Y)$-bigraph, $x\in X$ and $C$ is  a cycle  based on  $X-\{x\}$, then $x$ has at least two non-neighbors in $V(C) \cap Y$.
\end{lemma}

\begin{proof}
Let $|X|=k$ and let $C = x_1 y_1 \ldots x_{k-1} y_{k-1} x_1$. Suppose for the sake of contradiction that $|N(x) \cap V(C)| \geq  k-2$. If $N(x)$ contains a vertex $y$ that is not in $C$, then because $G$ is 2-connected,  $y$ has a neighbor in $V(C)$, say $x_1$. Then without loss of generality, $y_1 \in N(x)$, and we may replace the edge $x_1y_1$ in $C$ with the path $x_1 y x y_1$ to obtain a cycle  based on  $X$, a contradiction.

So we may assume $N(x) \subseteq V(C)$. Since $|\widehat{N}(X)| \geq |X|$, there exists a vertex $y \in \widehat{N}(X) \setminus V(C)$.
Since $G$ is $2$-connected and $yx\notin E(G)$,  $y$ has some neighbors $x_i$ and $x_j$ in $C$. If $\{y_{i+1}, y_{j+1}\} \subseteq N(x)$ then we obtain the cycle $x_1 y_1 \ldots x_i y x_j y_{j-1} \ldots y_{i+1} x y_{j+1} x_{j+2} \ldots x_1$, a contradiction. Similarly, we have that $\{y_{i-1}, y_{j-1}\}\not \subseteq N(x)$. The remaining case is $N(y) = \{x_i, x_{i+1}\}$ and $N(x) = V(C) - \{y_{i}\}$. By considering instead the cycle obtained by replacing $y_{i}$ with $y$, we see that by symmetry, $N(y_{i}) = \{x_i, x_{i+1}\}$. But this contradicts Lemma~\ref{y2}.
\end{proof} 

\section{$6$-cyclic graphs}\label{sec:k-6}

In this section, we complete the proof of Theorem~\ref{Xk} by proving that all $(X,Y)$-bigraphs satisfying~\eqref{lll} are $6$-cyclic.

\begin{lemma}\label{degree-4}
If $G$ is a saturated critical $(X,Y)$-bigraph and $|X|=6$, then $X$ contains a vertex of degree at least $4$.
\end{lemma}
\begin{proof}
Suppose all vertices in $X$ have degree at most $3$.

\textbf{Case 1:} There is a vertex $y \in Y$ with $d(y) \ge 4$. 
By Lemma~\ref{-2}, $d(y) \ne 4$, so $d(y) \ge 5$. Let $x_1, x_2, x_3, x_4, x_5$ be five neighbors of $y$; let $C = x_1y_1x_2y_2x_3y_3x_4y_4x_5y_5x_1$ be a cycle containing them.

If $y \notin V(C)$, then $x_1, x_2, x_3, x_4, x_5$ have two neighbors on $C$ and an edge to $y$, so they have degree $3$ and cannot have any other neighbors. In that case, the set $A = \{x_1, x_2, x_4\}$ contradicts~\eqref{lll}, since $\widehat{N}(A) = \{y_1, y\}$. 

Therefore $y \in V(C)$; say, $y = y_1$. Then $x_3, x_4, x_5$ have two neighbors on $C$ and an edge to $y_1$, so they have degree $3$. By~\eqref{lll} applied to $A = \{x_1, x_3, x_5\}$, $x_1$ must have an edge to one of $y_2$, $y_3$, $y_4$; symmetrically, $x_2$ must have an edge to one of $y_3$, $y_4$, $y_5$. This yields $3$ edges incident to each of $x_1, x_2, x_3, x_4, x_5$; none of these can have any other neighbors.

By~\eqref{lll},  $|\widehat{N}(X)| \geq 6$; however, since there is only one vertex in $X - V(C)$, $\widehat{N}(X) \subseteq N(X \cap V(C)) = Y \cap V(C)$. This only has size $5$, a contradiction.

\textbf{Case 2:} All vertices in $Y$ have degree at most $3$.
Let $X = \{x_1, x_2, x_3, x_4, x_5, x_6\}$. Let $C_1 = x_1y_1x_2y_2x_3y_3x_1$ be a $6$-cycle based on $\{x_1, x_2, x_3\}$ and let $C_2 = x_4y_4x_5y_5x_6y_6x_4$ be a $6$-cycle based on $\{x_4, x_5, x_6\}$. We have $V(C_1) \cap V(C_2) = \emptyset$, since a vertex in $V(C_1) \cap V(C_2)$ would have degree at least $4$.

In the cycle based on $\{x_1, x_2, x_4\}$, the vertex $x_4$ must have two common neighbors with $\{x_1, x_2\}$. Since $\Delta(G) \le 3$, at least one of them is a neighbor of $x_4$ on $C_2$. Without loss of generality, let $y_4$ be a common neighbor with $x_1$, so that $x_1y_4 \in E(G)$.

Now consider the cycle based on $\{x_1, x_4, x_5\}$. By the same argument, either $x_4$ or $x_5$ must be adjacent to one of $x_1$'s neighbors on $C_1$. Without loss of generality, let $x_4 y_1$ be that edge; then the cycle $x_1 y_4 x_5 y_5 x_6 y_6 x_4 y_1 x_2 y_2 x_3 y_3 x_1$ is based on  $X$, a contradiction. 
\end{proof}

\begin{claim}\label{x6}All  $(X,Y)$-bigraphs $G$ satisfying~\eqref{lll} are $6$-cyclic. 
\end{claim}

\begin{proof}
Take a vertex-minimal counterexample $G$ with the most edges, meaning in particular that $|X|=6$ and $Y = \widehat{N}(X)$. By Claims~\ref{x3}--\ref{x5}, $G$ is $k$-cyclic for $3 \le k \le 5$; therefore $G$ is critical, saturated and $Y$-minimal.

Let $X=\{x_1,\ldots,x_6\}$ and $x_6$ be a vertex of maximum degree in $X$. By Lemma~\ref{degree-4}, $d(x_6) \ge 4$. Let $C = x_1y_1x_2y_2x_3y_3x_4y_4x_5y_5x_1$ be a cycle based on $X - \{x_6\}$. By Lemma~\ref{two-neighbors} and Lemma~\ref{C-2}, $x_6$ has either $2$ or $3$ neighbors on $C$, so it has at least one neighbor $y_6$ not on $C$.

By symmetry, the following two cases are exhaustive.

\textbf{Case 1:} $\{y_1, y_3\} \subseteq N_C(x_6)$.
In this case, by Claim~\ref{no-con-2}, no vertex $y \in N(x_6) - V(C)$ can be adjacent to $x_1$, $x_2$, $x_3$, or $x_4$, so it must be adjacent to $x_5$ and $x_6$ only. By Lemma~\ref{y2}, $y_6$ is the only such vertex. By Claim~\ref{no-con-2} again, $x_6$ cannot be adjacent to $y_4$ or $y_5$, and to have degree $4$ it must also be adjacent to $y_2$.

If $x_2y_4 \in E(G)$, then the cycle $x_2y_4 x_4 \ldots y_2 x_6y_6 x_5 y_5x_1y_1x_2$ is based on  $X$, and if $x_2 y_5 \in E(G)$, the cycle $x_2 y_2 \ldots x_5 y_6 x_6 y_1 x_1y_5 x_2$ is based on  $X$. A similar argument shows that $x_3y_5, x_3y_4 \notin E(G)$. However, applying
Claim~\ref{x3} to $A = \{x_2, x_3, x_5\}$, we find distinct vertices $y' \in N(x_2) \cap N(x_5)$ and $y'' \in N(x_3) \cap N(x_5)$.  Therefore $x_5$ is adjacent to $y_4, y_5, y_6, y', y''$, and $d(x_5) \ge 5 > d(x_6)$, contradicting the choice of $x_6$.

\textbf{Case 2:} $N_C(x_6) = \{y_1, y_2\}$.
In this case, in order to have degree $4$, $x_6$ must have neighbors $y, y'$ outside $C$. By Claim~\ref{no-con-2}, $y$ and $y'$ can only have $x_4$ and $x_5$ as neighbors. If $y$ is adjacent to $x_4$ and $y'$ is adjacent to $x_5$, or vice versa, we contradict Claim~\ref{no-con-3}; if both are adjacent only to $x_4$ or both only to $x_5$, we contradict Lemma~\ref{y2}.
\end{proof}

\section{Bigraphs with high minimum degree}\label{sec:main-proof}

\subsection{Properties of smallest counterexamples}
Throughout this subsection, we assume that $G$ is a vertex-minimal counterexample to Theorem~\ref{mainpan3} with the most edges; let $G \in \cG(n,m,\delta)$. Then for each $X'\subset X$ with $X' \neq X$, $G[X' \cup Y]$ also satisfies the conditions of Theorem~\ref{mainpan3} and hence is super-cyclic.

Let $G'=G[X \cup \widehat{N}(X)]$, i.e., $G'$ is obtained by removing only the degree-$1$ vertices of $G$. Then $G'$ is critical and saturated. In particular, for every $x \in X$, there exists a cycle $C$, in $G'$ and therefore in $G$, such that $V(C)\cap X = X - \{x\}$. Because $G'$ is $2$-connected, there exists an $x, C$-{\em fan} $F$ of size $t \geq 2$, i.e., a set of $t$ paths from $x$ to $V(C)$ such that any two of them have only  $x$ in common.

Among all such triples $(C, x, F)$, choose a triple such that $t = |V(C) \cap V(F)|$ is maximized, and subject to this, $|V(F)|$ is minimized. Let 
$|V(C)| = 2\ell$ (so $|X| = \ell + 1)$. Fix a clockwise direction of $C$ and let $T = V(C) \cap V(F)=\{u_1,\ldots,u_t\}$. 

For every vertex $u$ of $C$, $x^+_C(u)$ (respectively,  $x^-_C(u)$) denotes the closest to $u$ clockwise (respectively, counterclockwise) vertex of $X$
distinct from $u$. 
For a set $U\subset V(C)$, $X_C^+(U)=\{x^+_C(u)\; : \;u\in U\}$. When $C$ is clear from the content, the subscripts could be omitted.
The vertices $ y^+(u),y^-(u)$ and the sets $X^-(U),Y^+(U),Y^-(U)$ are defined similarly. 

Viewing $F$ as a tree (spider) with root $x$, any two vertices $u,v\in V(F)$ define the unique $u,v$-path $F[u,v]$ in $F$. For $u,v\in V(C)$, let $C[u,v]$ be  the clockwise $u,v$-path in $C$ and let $C^-[u,v]$ be  the counterclockwise $u,v$-path in $C$.

\begin{lem}\label{tX'} $t \leq \ell-2$.
\end{lem}
\begin{proof}
We first show that
\begin{equation}
	\label{part-i}
	t \leq \ell - |T \cap X|.
\end{equation}
If $w\in T\cap X$ and $y^+(w)\in T$, then the cycle $wF[w,y^+(w)]y^+(w)C[y^+(w),w]w$ is  based on $X$, a contradiction. Similarly,
$y^-(w),x^+(w),x^-(w)\notin T$. Thus, $|T\cap X|\leq \ell/2$ and $|T\cap Y|\leq \ell-2|T\cap X|$. This proves~\eqref{part-i}.

For the remainder of the proof, note that if Claims~\ref{no-con-1}--\ref{no-con-3} are applied to $G'$, then the conclusions hold for $G$ as well, since they are unaffected by the addition of vertices of degree $1$ in $Y$.

Let $C =x_1 y_1 \ldots x_\ell y_\ell x_1$, and suppose  $t\geq \ell-1$.  By~\eqref{part-i}, $|T\cap X|\leq 1$. If  $T\cap X=\emptyset$, we may assume that $xy_i\in E(G)$ for all $1\leq i\leq \ell-1$.
By~\eqref{lll}, $|\widehat{N}(X)|\geq \ell+1$, so there is $y\in Y-V(C)$ with at least two neighbors in $X$.
This will contradict one of  Claims~\ref{no-con-1}--\ref{no-con-3} (possibly, in reversed orientation of $C$), unless all such $y$ are adjacent to only $x_{\ell}=x^-(y_\ell)$ and $x_1=x^+(y_\ell)$. Fix such a vertex $y$. Let $A = X - \{x_\ell\}$. There exists a vertex $y' \in (Y - V(C)) \cup \{y_{\ell}\}$ such that $y' \in \widehat{N}(A)$, i.e., $y'$ has two neighbors other than $x_{\ell}$ (so $y' \neq y$). Let $C'$ be the cycle obtained by replacing $y_{\ell}$ with $y$. Then the vertex $y'$ violates one of  Claims~\ref{no-con-1}--\ref{no-con-3} with respect to $C'$.

If $|T\cap X|= 1$, then by~\eqref{part-i}, we may assume that  $xy_i\in E(G)$ for all $1\leq i\leq \ell-2$ and that $x$ has a common neighbor $y\in Y-V(C)$ with $x_\ell$.
By~\eqref{lll}, $|\widehat{N}(X-x_\ell)|\geq \ell$, so there is $y_0\in (Y-V(C))\cup \{y_{\ell-1},y_\ell\}$ with at least two neighbors in $X-x_\ell$.
If $y_0\in (Y-V(C))$, this again will contradict one of  Claims~\ref{no-con-1}--\ref{no-con-3}, unless $N(y_0) = \{x_{\ell-1}, x_1\}$. In this case, we obtain the longer cycle $y_1C[y_1, y_{\ell-2}]y_{\ell-2}xyx_\ell y_{\ell-1}x_{\ell-1}y_0 x_1$. So suppose without loss of generality $y_0=y_\ell$ has a neighbor $z\in X-\{x_\ell,x_1\}$.
By the case, $z\neq x$, so suppose $z=x_j$ for some $2\leq j\leq \ell-1$. Then $G$ has cycle $y_\ell C[y_\ell,y_{j-1}]y_{j-1}xyx_\ell C^-[x_\ell,x_j]x_jy_\ell$  based on $X$,
a contradiction.
\end{proof}

Given a cycle $C$ and distinct $x_1,x_2,x_3\in X\cap V(C)$, we say that {\em $x_1$ and $x_2$ {\bf cross} at $x_3$} if the cyclic order is $x_1,x_3,x_2$ and 
$x_1y^+(x_3),x_2y^-(x_3)\in E(G)$ or if the cyclic order is $x_1,x_2,x_3$ and 
$x_1y^-(x_3),x_2y^+(x_3)\in E(G)$. In this case, we also say that {\em $x_3$ is {\bf crossed} by $x_1$ and $x_2$}.

The following is  Lemma 2.8 in~\cite{KLMZ}. It holds for each bipartite graph $G$ (no restrictions).

\begin{lem}[\cite{KLMZ}]\label{cros12}
Let $C$ be a cycle of an $(X,Y)$-bigraph $G$, and let $u, v\in V(C) \cap X$. If $u$ and $v$ have at most $a$ crossings, then $d_C(u) + d_C(v) \leq |V(C)|/2 + 2 + a$.\end{lem}

\begin{lem}\label{T+} If $u_i\in X\cap T$, then
 $y^+(u_i)$ has no neighbors in $(F-V(C)) \cup X^+(T)$.
\end{lem}
\begin{proof} Suppose $y^+(u_i)$ has a neighbor $z$ in $F-V(C)$. Then the cycle $u_iF[u_i,z]zy^+(u_i)C[y^+(u_i),u_i]u_i$ is based on  $X$, a contradiction.

Suppose now that $y^+(u_i)$ has a neighbor $x_1$ in $X^+(T)$, where $u \in T$ satisfies $x^+(u) = x_1$. Then the cycle $x_1 y^+(u_i) C[y^+(u_i), u] u F[u, u_i] u_i C^-[u_i, x_1] x_1$ is based on  $X$, a contradiction.
\end{proof}

\begin{lem}\label{neighborF0}
If $x_1\in X^+(T)$, then $x_1$ cannot have a neighbor in $F-V(C)$.
\end{lem}
\begin{proof}
Suppose $x_1$ has a neighbor $y'$ in $F-V(C)$. Let $u_1\in T$ be such that $x_1=x^+(u_1)$ and $z$ be a neighbor of $u_1$ in $F$. Let $P$ be
a $z,y'$-path in $F$ and the cycle $C'$ be defined by  $C'=x_1C[x_1,u_1]u_1zPy'x_1$. If $y'\neq z$, then $C'$
 is  based on $X$ and we are done. Thus $z=y'$ and hence $u_1\in X$. If $y'=z$, then let $F'=F-u_1$. Note $F'$ is a fan of $C'$ such that $|V(F\cap C)|=|V(F'\cap C')|$, but $|V(F')|<|V(F)|$, contradicting the choice of $C$ and $F$.
\end{proof}

%

\begin{lem}\label{cross0} Suppose that $x_1,x_2\in X^+( T)$. Then
\begin{enumerate}
\item[(i)] $x_1$ and $x_2$ share no neighbors in $Y - V(C)$;
\item[(ii)] neither $x_1$ nor $x_2$ share a neighbor in $Y - V(C)$ with $x$.
\end{enumerate}
\end{lem}

\begin{proof} From Lemma~\ref{neighborF0}, we know that if $x_1$ and $x$ have a common neighbor outside of $C$, it is not in $F$. Suppose they share some neighbor $y \in Y - V(C)$. Let $x_1=x^+(u_1)$. Then we have a longer cycle $x_1C[x_1,u_1]u_1F[u_1,x]xyx_1$. This proves (ii).

Suppose $x_1$ and $x_2$ share a neighbor $y \in Y - V(C)$, and
 $u_1,u_2\in  T$ are such that $x_1=x^+(u_1)$ and $x_2=x^+(u_2)$.
 By Lemma~\ref{neighborF0}, $y\notin F$. The cycle 
$C' := x_1C[x_1,u_2]u_2F[u_2,u_1]u_1C^-[u_1,x_2]x_2yx_1$ is  based on $X$, a contradiction.
\end{proof}


\begin{lem}\label{cross1} Suppose $u_1,u_2\in {T}$.
If $x_1=x^+(u_1)$ and $x_2=x^+(u_2)$ cross at $x_3\in X\cap V(C)$, then
\begin{enumerate}
\item[(i)] $x_3\notin {T}$;
\item[(ii)] $G$ has a cycle $C'$ containing $(X\cap V(C)-\{x_3\})\cup \{x\}$ such that  $|C'|\geq |C|$;
\item[(iii)] $x_3$ shares no neighbors in $Y-V(C)$ with any vertex in the set $\{x\}\cup X^+(T)$;
\item[(iv)] $x_3$ has at most $t$ neighbors on $C$.
\end{enumerate}
\end{lem}
\begin{proof}
For part (i), suppose that the cyclic order is $x_1,x_3,x_2$ and  $x_1y^+(x_3),x_2y^-(x_3)\in E(G)$ (the other case is symmetric). Let $y$ be a neighbor of $x_3$ in $F$. 
Let $z$ be a neighbor of $u_1$ in $F$. Let $P$ be
a $z,y$-path in $F$ and the cycle $C'$ be defined by 
\[
	C' := x_1y^+(x_3)C[y^+(x_3),u_1]u_1zPyx_3C^-[x_3,x_1]x_1.
\]
Note $C'$
 is  based on $X$. 

The cycle 
\[
	C_1 := x_1y^+(x_3)C[y^+(x_3),u_2]u_2F[u_2,u_1]u_1C^-[u_1,x_2]x_2y^-(x_3)C^-[y^-(x_3),x_1]x_1
\]
proves (ii).

To prove (iii), assume that $y \in Y -V(C)$ is a common neighbor of $x_3$ and a vertex in $\{x\}\cup X^+(T)$, and
consider all cases. If $yx \in E(G)$, let $$ C'=x_1y^+(x_3)C[y^+(x_3),u_1]u_1F[u_1,x]xyx_3C^-[x_3,x_1]x_1.$$ If $y$ is not in $F[x,u_1]$, then $C'$ is a cycle  based on $X$, a contradiction. Otherwise, let $F''$ be $F-F[u_1,y]$. Note $F''$ is a fan of $C''=x_1y^+(x_3)C[y^+(x_3),u_1]u_1F[u_1,y]yx_3C^-[x_3,x_1]x_1$ such that $|V(F\cap C)|=|V(F''\cap C'')|$, but $|V(F'')|<|V(F)|$, contradicting the choice of $C$ and $F$.

If $u_j\in T$, $x_j=x^+(u_j)$, $yx_j\in E(G)$, and   $x_j\in C[y^+(x_3),u_1]$, then the cycle 
\[
	C' := x_1C[x_1,x_3]x_3yx_jC[x_j,u_1]u_1F[u_1,u_j]u_jC^-[u_j,y^+(x_3)]y^+(x_3)x_1
\]
is  based on $X$. Similarly, if $x_j\in C[u_1,y^-(x_3)]$, then the cycle
\[
	C' := x_2C[x_2,u_j]u_jF[u_j,u_2]u_2C^-[u_2,x_3]x_3yx_jC[x_j,y^-(x_3)]y^-(x_3)x_2
\]
is  based on $X$,
 a contradiction. This proves (iii).

By the choice of $(C,x,F)$ and (ii), $x_3$ has at most $t$ neighbors on $C_1$. The only vertices in $Y\cap V(C)-V(C_1)$ are $y^-(x_1)$ and $y^-(x_2)$. If $x_3y^-(x_1)\in E(G)$, then the cycle 
\[
	y^-(x_1)C[y^-(x_1),y^-(x_3)]y^-(x_3)x_2C[x_2,u_1]u_1F[u_1,u_2]u_2C^-[u_2,x_3]x_3y^-(x_1)
\]
is  based on $X$. If $x_3y^-(x_2)\in E(G)$, then the cycle 
\[
	x_1C[x_1,x_3]x_3y^-(x_2)C[y^-(x_2),u_1]u_1F[u_1,u_2]u_2C^-[u_2,y^+(x_3)]y^+(x_3)x_1
\]
is  based on $X$. 
This proves (iv).
\end{proof}


\begin{lem}\label{cros11} Suppose $u_1,u_2\in {T}$,
 $x_1=x^+(u_1)$, and $x_2=x^+(u_2)$. Then no two vertices   $x_3,x_4\in V(C)$  crossed by $x_1$ and $x_2$ have a shared neighbor in $Y - V(C)$.
\end{lem}
\begin{proof}
Suppose vertices $x_3, x_4 \in V(C) \cap X$ are crossed by $x_1$ and $x_2$ and there is some $y \in (N(x_3) \cap N(x_4)) - V(C)$. By Lemma~\ref{cross1}, $y \notin V(F)$.

We consider two cases. If $x_3$ and $x_4$  both are on $C[x_1,x_2]$ or both are on $C[x_2, x_1]$, then we may assume that their  cyclic order is $x_1, x_3, x_4, x_2$. In this case, the cycle 
\[
	x_1 C[x_1, x_3] x_3 y x_4 C[x_4, u_2] u_2 F[u_2, u_1] u_1C^-[u_1, x_2] x_2y^-(x_4) C^-[y^-(x_4), y^+(x_3)] y^+(x_3)x_1
\] is  based on $X$.

%

If one of $x_3$ and $x_4$ is  on $C[x_1, x_2]$ and the other is on $C[x_2, x_1]$, then we may assume that their cyclic order is $x_1, x_3, x_2, x_4$. 
In this case, the cycle
\[
	x_1 C[x_1, x_3] x_3 y x_4 C^-[x_4, x_2] x_2 y^+(x_4) C[y^+(x_4), u_1] u_1 F[u_1, u_2] u_2C^-[u_2, y^+(x_3)] y^+(x_3) x_1
\] is  based on $X$.
This proves the lemma. 
\end{proof}


\begin{lem}\label{smallsum}Let $A \subseteq X^+(T)$. Then $\sum_{w \in A} d_C(w) \leq |A|(\ell - 2) + 2$. 
\end{lem}
\begin{proof}
Let $x_1, x_2 \in A$ such that $x_1= x^+(u_1)$ and $x_2 = x^+(u_2)$ for some $u_1, u_2 \in T$. We first prove that 
\begin{equation}\label{yneighbor2}\mbox{if $u_2 \in Y$ and $y^+(x_2)x_1 \in E(G)$, then $d_C(x_2) \leq \ell - 2$. }
\end{equation}
The cycle $C' = x_1 y^+(x_2) C[y^+(x_2), u_1]u_1 F[u_1, u_2]u_2 C^-[u_2, x_1] x_1$ contains all vertices in $C$ except $x_2$ and possibly $y^+(u_1)$ (in the case that $u_1 \in X$).  By Lemma~\ref{T+}, $N_C(x_2) = N_{C'}(x_2)$.  By Lemma~\ref{tX'} applied to $C'$ and $x_2$, $d_{C}(x_2) = d_{C'}(x_2) \leq \ell - 2$. 

In particular, if $d_C(x_1) = \ell$, i.e., $x_1$ is adjacent to every $y$ vertex in $C$, then by Lemma~\ref{T+}, each $x_2 \in X^+(T) - \{x_1\}$ satisfies $u_2 \in Y$. Therefore by~\eqref{yneighbor2}, $d_C(x_2) \leq \ell - 2$. It follows that $\sum_{w \in A} d_C(w) \leq |A|(\ell - 2) + 2$. 

So suppose every $w \in A$ has $d_C(w) \leq \ell - 1$, and there exist two vertices $x_1, x_2 \in A$ with equality. Define $u_1, u_2$ as before. Then for every $x_3 \in X^+(T) - \{x_1,x_2\}$, either $u_1 \in Y$ and $x_3y^+(x_1) \notin E(G)$ by~\eqref{yneighbor2}, or $u_1 \in X$ and $x_3y^+(x_1) \notin E(G)$ by Lemma~\ref{T+}. The same holds for $x_3$ and $x_2$. Therefore $d_C(x_3) \leq \ell-2$, and again $\sum_{w \in A} d_C(w) \leq |A|(\ell - 2) + 2$. 
\end{proof}

\begin{lem}\label{cros111} Suppose $t\geq 4$, $u_1,u_2\in {T}$,
 $x_1=x^+(u_1)$, and $x_2=x^+(u_2)$. 
 Then  at most one vertex in $C$ is crossed by $x_1$ and $x_2$.
\end{lem}
\begin{proof}
Suppose vertices $x_3, x_4 \in V(C) \cap X$ are crossed by $x_1$ and $x_2$.

 Let $A=X^+(T)\cup\{x,x_3,x_4\}$
(possibly, $X^+(T)\cap\{x_3,x_4\}\neq \emptyset$), and $A'=A-\{x,x_3,x_4\}$. Note that $|A'|\geq t-2$, and by Lemma~\ref{smallsum} applied to $A'$, $\sum_{w \in A'}d_C(w) \leq |A'|(\ell-2)+2$.

Since $x$ can have at most $t$ neighbors on $C$, $|N(x) - V(C)| \geq \delta-t$. By Lemma~\ref{cross1}(iv),  $|N(x_3)-V(C)|\geq \delta-t$ and $|N(x_4)-V(C)|\geq \delta-t$. 
By Claims~\ref{no-con-1}--\ref{no-con-3} (applied to $G'$) and Lemmas~\ref{cross1}(iii) and~\ref{cros11}, no two distinct vertices in $A$ have a common neighbor in $Y - V(C)$.
Thus, remembering the $\ell$ vertices in $Y\cap V(C)$, we get \allowdisplaybreaks
\begin{align*}
|Y|	&\geq \ell+\sum_{u\in A}|N(u)-V(C)|\\
	&= \ell+|N(x)-V(C)|+|N(x_3)-V(C)|+|N(x_4)-V(C)|+\sum_{u\in A'}|N(u)-V(C)|\\
	&\geq \ell+3(\delta-t)+\delta |A'|- (\ell-2)|A'|-2
	\geq \ell+3\delta-3t+(\delta-\ell+2)|A'|-2\\
		&	\geq \ell+3\delta-3t+(\delta-\ell+2)(t-2)-2\
		\geq \ell+3\delta-3t+(\delta-\ell+2)+(\delta-\ell+2)(t-3)-2\\
	&\geq  \ell+3\delta-3t+(\delta-\ell+2)+3(t-3)-2
	= 4\delta -3t+2+3(t-3)-2
		=4\delta-9,
\end{align*}
as claimed.
\end{proof}

\begin{lem}\label{3con}
For any $x_1, x_2 \in X$, $x_1$ and $x_2$ cannot be separated by a set of two vertices.
\end{lem}

\begin{proof}
Recall that $G$ is a vertex-minimum counterexample to Theorem~\ref{mainpan3}, and $G' = G[X \cup \widehat{N}(X)]$ is critical and saturated.

Suppose that for some $x_1, x_2 \in X, u_1, u_2 \in V(G)$, $x_1$ and $x_2$ are in different components of $G - \{u_1, u_2\}$. Note that $u_1, u_2 \in V(G')$, since $V(G) - V(G')$ contains only vertices of degree $1$ in $Y$.

If there also exists $x_3 \in X - \{x_1, x_2\}$ such that $x_3$ is in a different component of $G - \{u_1, u_2\}$ than both $x_1$ and $x_2$, then $G$ cannot contain a 6-cycle based on $\{x_1, x_2, x_3\}$, since these vertices are separated by a set of size two. Hence we may assume $G - \{u_1, u_2\}$ contains exactly two components containing vertices in $X$. Call these components $D_1$ and $D_2$ where $x_1 \in V(D_1)$ and $x_2 \in V(D_2)$.

Choose any two vertices $x, x' \in X - \{u_1, u_2\}$; then choose a third vertex $x'' \in X - \{u_1, u_2\}$ such that not all three of $x, x', x''$ are in the same component of $G - \{u_1, u_2\}$. Let $C$ be a cycle based on $A = \{x, x', x''\}$.

Since $\{u_1, u_2\}$ separates one of the vertices of $A$ from the others, $u_1, u_2 \in V(C)$; since $V(C) \cap X = A$ and neither $u_1$ nor $u_2$ is in $A$, we must have $u_1, u_2 \in Y$.

Moreover, $u_1$ must have an edge to either $x$ or $x'$ in $C$, and therefore in $G$. Since $x, x' \in X$ were arbitrary, $|N(u_1)| \ge |X|-1$. By Lemma~\ref{5c} applied to $G'$, $N_{G'}(u_1) = X$, and therefore $N(u_1) = X$. By symmetry, we also obtain $N(u_2) = X$.

Now suppose each component of $G - \{u_1, u_2\}$ has at least 2 vertices in $X$. For $i \in \{1,2\}$, set $X_i = X \cap D_i$. By the minimality of $G$, there exists a cycle $C_1$ of $G$ based on $X_1 \cup \{x_2\}$ and a cycle $C_2$ based on $X_2 \cup \{x_1\}$. Since $D_1$ and $D_2$ are separated by $\{u_1, u_2\}$, $N_{C_1}(x_2) = N_{C_2}(x_1) = \{u_1,u_2\}$. Therefore $(C_1 - \{x_2\}) \cup (C_2 - \{x_1\})$ is a cycle in $G$ which is based on  $X$, a contradiction.

Thus we may assume without loss of generality that $V(D_1) \cap X = \{x_1\}$. Note that this implies all neighbors of $x_1$ other than $u_1$ and $u_2$ have degree 1. Let $G_1$ be obtained from $G$ by deleting $x_1$ and all of its neighbors except for $u_1$.

We will show that $G_1$ is a counterexample that has fewer vertices than $G$. Set $X' = X - \{x_1\} = X \cap V(G_1)$. If there exists $A \subseteq X'$ with $|A| \geq 3$ such that $|\widehat{N}_{G_1}(A)| < |A|$, then in $G$, $\widehat{N}_{G}(A \cup \{x_1\}) = \widehat{N}_{G_1}(A) \cup \{u_2\} < |A \cup \{x_1\}|$, a contradiction.

Next, we will show that for all $A$ with $|A| \geq 3$, $G_1[A \cup \widehat{N}_{G_1}(A)]$ is 2-connected. Recall that $G_1 - \{u_1\} = D_2$. The subgraph of $D_2$ obtained by removing all vertices in $Y$ of degree 1 is still connected. Call this subgraph $H$. If $A = X'$, then $G_1[A \cup \widehat{N}_{G_1}(A)] = G_1[H \cup \{u_1\}]$. Since $H$ is connected and $u_1$ is adjacent to all vertices in $X'$, $G_1[H \cup \{u_1\}]$ is 2-connected. Now suppose $A \neq X'$. Then by the choice of $G$ as a minimum counterexample, there exists a cycle $C$ in $G$ with $V(C) \cap X = A \cup \{x_1\}$, where $N_C(x_1) = \{u_1, u_2\}$. In particular,  $P:=C - \{x_1, u_1, u_2\}$ is a path containing all vertices of $A$. In $G_1$, $G_1[A \cup \widehat{N}_{G_1}(A)]$ can be obtained from $P$ by adding $u_1$, which is adjacent to all of $V(P) \cap X$, and possibly adding some additional vertices in $Y$ with degree at least 2. Hence it is 2-connected.

Next, suppose  that it is super-cyclic. By the minimality of $G$, $G$ contains no cycle $C$ based on $X$; however, because $G_1$ is super-cyclic, we may find a cycle $C' = v_1v_2 \ldots v_{2|X'|}v_1$ in $G_1$ (and therefore in $G$) based on $X - \{x_1\}$ such that $u_2 \notin V(C')$. If $u_1 \notin V(C')$, then we may replace in $C'$ any segment $v_i v_{i+1} v_{i+2}$ (for $v_i \in X$) with the path $v_iu_1x_1 u_2 v_{i+2}$ to obtain a contradiction. Otherwise, if $u_1 = v_i$ for some $i$, we replace the path $v_{i-1} u_1 v_{i+2}$ with $v_{i-1} u_1 x u_2 v_{i+2}$. 

Finally, we have $|Y \cap G_1| \leq |Y| - (\delta - 1) \leq (4\delta - 10) - (\delta - 1) \leq 4 (\delta-1) - 10$. The last inequality holds because we may assume that $|X| \ge 7$ and therefore $\delta \ge 7$, since Theorem~\ref{Xk} implies Theorem~\ref{mainpan3} for $|X| \le 6$. This shows that $G_1$ is a counterexample for Theorem~\ref{mainpan3} (with $\delta' = \delta-1$) which has fewer vertices than $G$, contradicting the choice of $G$.
\end{proof}

\subsection[Proof of Theorem 1.8]{Proof of Theorem~\ref{mainpan3}}

\begin{proof}[Proof of Theorem~\ref{mainpan3}]
As in the previous subsection, suppose for the sake of contradiction that $G$ is a vertex-minimum, edge-maximal counterexample to Theorem~\ref{mainpan3}. By the choice of $G$, for each $x \in X$, there exists some cycle $C$ with $V(C) \cap X = X - \{x\}$. We may also assume that $|X| \ge 7$ and therefore $\delta \ge 7$, since Theorem~\ref{Xk} implies Theorem~\ref{mainpan3} for $|X| \le 6$.

Letting $G' = G[X \cup \widehat{N}(X)]$, it follows from our choice of $G$ that $G'$ is critical and saturated.

If there exists a pair $(x, C)$ with an $x,C$-fan $F$ of size at least $4$, then choose such a triple which maximizes $t = |V(F) \cap V(C)|$, and subject to this, minimizes $|V(F)|$. Let $T = V(F) \cap V(C)$. By Lemmas~\ref{cross0} and~\ref{cros111}, no two vertices in $X^+(T) \cup \{x\}$ share a neighbor outside of $V(C)$, and no two vertices in $X^+(T)$ cross more than one time. By Lemma~\ref{cros12}, for each pair $x_1, x_2 \in X^+(T)$, $d_C(x_1) + d_C(x_2) \leq |V(C) \cap Y| + 3 = |X|+2$.
 Therefore 
\begin{eqnarray*}
|Y|  & \geq & |V(C)\cap Y| + \sum_{w \in X^+(T) \cup \{x\}}d_{Y - V(C)}(w)\\&\geq & |X| - 1 + \delta (t+1) - \sum_{w \in X^+(T) \cup \{x\}} d_C(w) \\
& \geq & |X| - 1 + \delta (t+1) - t - t(|X|+2)/2.
\end{eqnarray*}Since the coefficient at $t$ is at least $\delta - 1 - (\delta + 2)/2 > 0$ (assuming, as we do, that $\delta > 4$), this quantity is minimized whenever $t$ is minimized, i.e., $t = 4$. We obtain $|Y| \geq |X| - 1 + 5\delta - 4 - 2(|X|+ 2)$, which is minimized when $|X| = \delta$. So $|Y| \geq 4\delta - 9$, a contradiction.

Now suppose that for all $x \in X$ and cycles $C$ with $V(C) \cap X = X - \{x\}$, the largest $x, C$-fan has size at most 3. Choose $x \in X$ with the maximum number of neighbors of degree at least 2. If every $x \in X$ has at most 3 neighbors of degree at least 2 (and at least $\delta - 3$ neighbors of degree 1), then we have
$|Y| \geq |X|(\delta - 3)+3$; since $|X| \geq 4$, $|Y| \geq 4\delta - 9$.

Therefore $x$ has at least 4 neighbors of degree at least 2. Let $F$ be a maximum $x, C$-fan of $G$ and set $T = F \cap V(C)$. By Lemma~\ref{3con}, $|T| \geq 3$, since $x$ cannot be separated from $X - x'$ by a set of size 2. So $|T|=3$.

By Lemma~\ref{two-neighbors}, $|T \cap Y|\geq 2$ (we apply this lemma to $G'$, but the conclusion carries over to $G$). If $|T \cap Y| = 3$, then since $x$ has at least 4 neighbors of degree at least 2, there exists $y \in N(x) - V(C)$. Since all vertices in $X - \{x\}$ are contained in $C$, $y$ has a neighbor $x' \in C$. But then $F \cup xyx'$ is a fan of size 4, a contradiction.

Finally, we may assume $T \cap V(C) = \{x_1, y_1, y_2\}$, where $x$ and $x_1$ have at least $2$ common neighbors outside $C$. In particular, $\{y_1, y_2\} \subset N(x)$, and for any $x' \neq x,x_1$, we have $N(x) \cap N(x') \subseteq \{y_1,y_2\}$, otherwise we could find a larger $x, C$-fan.  We will show that $N(y_1) = N(y_2) = X$. If there exists $x', x'' \in X - \{x_1\}$ such that $x'y_i, x''y_i \notin E(G)$ for some $i \in \{1,2\}$, then there cannot exist a 6-cycle based on $\{x, x', x''\}$, a contradiction. Hence $|N(y)| \geq |X|-2$ which implies $N(y) = X$ by Lemma~\ref{-2} (again, we apply this lemma to $G'$, but the conclusion carries over to $G$).

Consider $y \in \widehat{N}(x) -\{y_1,y_2\}$. Since there is no $x,C_x$-fan of size 4, $N(y) \subseteq T \cup \{x\}$. That is, $N(y) = \{x, x_1\}$ and so $\widehat{N}(x) \subseteq N(x_1)$. By the choice of $x$, $N_C(x_1) = \{y_1,y_2\}$ since otherwise $|\widehat{N}(x_1)| > |\widehat{N}(x)|$. But then $\{y_1,y_2\}$ separates $\{x,x_1\}$ from the rest of the vertices in $X$, contradicting  Lemma~\ref{3con}.
\end{proof}

\end{document}